\documentclass[11pt]{amsart}
\usepackage{amssymb,amsmath,cite,hyperref}
\usepackage[usenames]{color} \newtheorem{theorem}{Theorem}[section]
\usepackage{pgfplots}
\pgfplotsset{compat=1.15}
\usepackage{mathrsfs}
\usetikzlibrary{arrows}

\newtheorem{proposition}[theorem]{Proposition}
\newtheorem{lemma}[theorem]{Lemma}

\newtheorem{corollary}[theorem]{Corollary}

\title{On the subadditivity condition of edge ideal}
\author{Abed Abedelfatah}
\address{Department of Mathematics, Braude College of Engineering, 2161002 Karmiel, Israel}
\email{abed@braude.ac.il}
\begin{document}
\maketitle
\begin{abstract}
Let $S=K[x_1,\ldots,x_n]$, where $K$ is a field, and $t_i(S/I)$ denotes the maximal shift in the minimal graded free $S$-resolution of the graded algebra $S/I$ at degree $i$, where $I$ is an edge ideal. In this paper, we prove that if $t_b(S/I)\geq \lceil \frac{3b}{2} \rceil$ for some $b\geq 0$, then the subadditivity condition $t_{a+b}(S/I)\leq t_a(S/I)+t_b(S/I)$ holds for all $a\geq 0$. In addition, we prove that $t_{a+4}(S/I)\leq t_a(S/I)+t_4(S/I)$ for all $a\geq 0$ (the case $b=0,1,2,3$ is known). We conclude that if the projective dimension of $S/I$ is at most $9$, then $I$ satisfies the subadditivity condition.
\end{abstract}

\section{Introduction}
Let $S=K[x_1,\ldots,x_n]$ denote the polynomial ring with $n$ variables over the fixed field
$K$, graded by setting $\deg(x_i)=1$ for each variable. Since Hilbert’s syzygy theorem,
minimal free resolutions of graded finitely generated $S$-modules, and particularly their
graded Betti numbers, which carry most of the numerical data
about them, became central invariants of study in commutative algebra, with
applications in other areas, e.g., in algebraic geometry, hyperplane arrangements, and combinatorics.
Restricting to $S$-modules $S/I$ for monomial ideals $I$, and particularly to edge ideals, makes combinatorial,
and particularly graph theoretical, tools available. This perspective proved to be very useful in recent decades.

Assume that $I$ is a graded ideal of $S$ and suppose $S/I$ has minimal graded free $S$-resolution
$$0\rightarrow F_p=\bigoplus_{j\in \mathbb{N}}S(-j)^{\beta _{p,j}}\rightarrow\cdots\rightarrow F_1=\bigoplus_{j\in \mathbb{N}}S(-j)^{\beta _{1,j}}\rightarrow F_0=S\rightarrow S/I\rightarrow0.$$
The numbers $\beta_{i,j}=\beta_{i,j}(S/I)$, where $i,j\geq0$, are called the \emph{graded Betti numbers} of $I$, which count the elements of degree $j$ in a minimal generator of ($i+1$)-th \emph{syzygy}:~ $\mathrm{Syz}_{i+1}(S/I)=\ker~(F_i\rightarrow F_{i-1})$. Let $t_i$ denote the maximal shifts in the minimal graded free $S$-resolution of $S/I$, namely
\[t_i=t_i(S/I):= \max\{j:\ \beta_{i,j}(S/I)\neq 0\}.\]
We say that $I$ satisfies the \emph{subadditivity condition} if
\begin{equation}\label{eq2}
t_{a+b}\leq t_a+t_b
\end{equation}
for all $a,b\geq0$ and $a+b\leq \mathrm{pd}_S(S/I)$, where $\mathrm{pd}_S(S/I)$ is the projective dimension of $S/I$.

It is known that graded ideals may not satisfy the subadditivity condition as shown by the counter example in \cite[Section 6]{Conca-Sub}. However, no counter examples are known for monomial ideals. Note that if $I$ is a monomial ideal, then by polarization, we can reduce the problem to a squarefree monomial ideal.
In the case of monomial ideals, there are special cases for which the subadditivity condition holds: Herzog and Srinivasan \cite[Corollary 4]{Herzog-Srinivasan} proved (\ref{eq2}) for $b=1$ which was proved earlier for edge ideals in \cite[Theorem 4.1]{Oscar}. The author and Nevo proved (\ref{eq2}) when $b=2,3$ for edge ideals \cite[Theorem 1.3]{Abed-Eran}. Bigdeli and Herzog proved (\ref{eq2}), for edge ideal of a chordal graph or a whisker graph\cite[Theorem 1]{mina-herzog}. The author proved (\ref{eq2}) when $I=I_{\Delta}$ where $\Delta$ is a simplicial complex such that $\dim(\Delta)< t_b-b$ \cite[Theorem 3.3]{Abed}. Jayanthan and Kumar proved (\ref{eq2}) for several classes of edge ideals of graphs and path ideals of rooted trees \cite{Jayanthan-Kumer}. Some more results regarding subadditivity condition have been obtained by Khoury and Srinivasan \cite[Theorem 2.3]{Khoury-Srin}, Faridi and Shahada \cite[Corollarly 5.3]{Faridi-Shahada} and Faridi \cite[Theorem 3.7]{Faridi}.

Let $I$ be an edge ideal of graph $G$. By using topological-combinatorics arguments, we prove that, over any field, if $t_b(S/I)\geq \lceil \frac{3b}{2} \rceil$ for some $b\geq 0$, then for all $a\geq0$, $t_{a+b}(S/I)\leq t_a(S/I)+t_b(S/I)$. Note that in general $b<t_b\leq 2b$. In addition, we prove the subadditivity condition (\ref{eq2}) for $0\leq b\leq 4$ and all $a\geq0$.
We conclude that if $t_j\geq \lceil \frac{3j}{2} \rceil$ for all $5\leq j\leq \lfloor \frac{\mathrm{pd}_S(S/I)}{2}\rfloor$, then $I$ satisfies the subadditivity condition. In particular, if $\mathrm{pd}_S(S/I)\leq 9$, then $I$ satisfies the subadditivity condition.\\

\section{Preliminaries }
Fix a field $K$.
Let $S=K[x_1,\dots,x_n]$ be the graded polynomial ring with $\deg(x_i)=1$ for all $i$, and $M$ be a graded $S$-module. The integer $\beta_{i,j}^S(M)=\dim_KTor_i^S(M,K)_j$ is called the $(i,j)$\emph{th graded Betti number of} $M$. Note that if $I$ is a graded ideal of $S$, then $\beta_{i+1,j}^S(S/I)=\beta_{i,j}^S(I)$ for all $i,j\geq 0$.

For a simplicial complex $\Delta$ on the vertex set $\Delta_0=[n]=\{1,\dots,n\}$, its \emph{Stanley-Reisner ideal} $I_{\Delta}\subset S$ is the ideal generated by the squarefree monomials $x_F=\prod_{i\in F}x_i$ with $F\notin \Delta$, $F\subset [n]$. The dimension of the face $F$ is $|F|-1$ and the \emph{dimension of $\Delta$} is $\max\{\dim F~:~F\in\Delta\}$.
Let $G$ be a simple graph on the set $[n]$ and denote by $E(G)$ the set of its edges. We define the \emph{edge ideal} of $G$ to be the ideal
\[
I(G)=\langle x_ix_j~:~\{i,j\}\in E(G)\rangle\subset S.
\]
So if $\Delta$ is a flag simplicial complex and $H$ is the graph of minimal non-faces of $\Delta$, then $I_\Delta=I(H)$.

For $W\subset V$, we write $$\Delta[W]=\{F\in\Delta~:~F\subset W\}$$ for the induced subcomplex of $\Delta$ on $W$.
We denote by $\beta_i(\Delta)=\dim_K \widetilde{H}_i(\Delta;K)$ the dimension of the $i$-th reduced homology group of $\Delta$ with coefficients in $K$.
The following result is known as Hochster's formula for graded Betti numbers.

\begin{theorem}[Hochster]
Let $\Delta$ be a simplicial complex on $[n]$. Then
$$\beta_{i,i+j}(S/I_{\Delta})=\sum_{W\subset[n],~|W|=i+j}\beta_{j-1}(\Delta[W];K)$$ for all $i,j\geq0$.
\end{theorem}

If $\Delta_1$ and $\Delta_2$ are two subcomplexes of $\Delta$ such that $\Delta=\Delta_1\cup \Delta_2$, then there is a long exact sequence of reduced homologies, called the \emph{Mayer-Vietoris sequence}
\begin{align*}
\cdots \rightarrow \widetilde{H}_i(\Delta_1\cap\Delta_2;K)\rightarrow \widetilde{H}_i(\Delta_1;K)\oplus \widetilde{H}_i(\Delta_2;K)&\rightarrow \widetilde{H}_i(\Delta;K)\\\rightarrow \widetilde{H}_{i-1}(\Delta_1\cap\Delta_2;K)\rightarrow\cdots
\end{align*}

Using the Mayer-Vietoris sequence, Fern{\'a}ndez-Ramos and Gimenez proved the following lemma that we will use in the main results.
\begin{lemma}\emph{(\hspace{1sp}\cite[Theorem 2.1]{Oscar})}\label{thm-corner}
For an edge ideal $I=I(G)$, over any field, if $\beta_{i,i+j}(S/I)=0=\beta_{i,i+j+1}(S/I)$ then $\beta_{i+1,i+j+2}(S/I)=0$.
\end{lemma}

Finally we state the following results of the author and Nevo on vanishing patterns in the Betti table of edge ideals and the subadditivity condition.

\begin{lemma}\emph{(\hspace{1sp}\cite[Theorem 3.4]{Abed-Eran})}\label{thm-vanishing}
For an edge ideal $I=I(G)$, over any field, if $\beta_{i,i+2}(S/I)\neq 0$ and $\beta_{i+k,j+2+k}(S/I)\neq 0$ where $i\geq 0$ and $k>0$, then $$\beta_{i+m,j+2+m}(S/I)\neq 0$$ for all $0\leq m\leq k$.
\end{lemma}

\begin{lemma}\emph{(\hspace{1sp}\cite[Theorem 1.3]{Abed-Eran})}\label{thm-abederan}
For any edge ideal over any field, the subadditivity condition (\ref{eq2}) holds for $b=1,2,3$ and any  $a\geq 0$.
\end{lemma}

\section{The main Results}
We start with the following lemma that we will use in the main results.
\begin{lemma}\label{lem 1}
If $I=I(G)$ is the edge ideal of a graph $G$ and $t_a<2a$ for $a\geq 2$, then for each $r$ edges $\{v_1,v_2\},\dots,\{v_{2r-1},v_{2r}\}$ in $G$ where $r>a$, the induced subgraph $H$ on the vertices $v_1,\dots,v_{2r}$ has a vertex $v$ with $$\deg_H(v)\geq2.$$
\end{lemma}

\begin{proof}
Assume on contrary that $\deg_H(v)=1$ for all $v\in H$. It follows that $\beta_{a,2a}(S/I(H))\neq 0$ since $\beta_{a,2a}(S/I(H))$ is the number of induced subgraphs of $H$ which are $a$ disjoint edges. By Hochster's formula $$\beta_{a,2a}(S/I(H))\leq \beta_{a,2a}(S/I(G)),$$ so $\beta_{a,2a} (S/I(G))\neq 0$, a contradiction to the assumptions.
\end{proof}

\begin{proposition}\label{pro1}
Over any field, if $I=I(G)$ is the edge ideal of a graph $G$ and $t_a<2a$ for some $a\geq 2$, then for all $b\geq0$,
$$t_{a+b}\leq t_a+\lceil \frac{3b}{2} \rceil.$$
\end{proposition}

\begin{proof}
We prove the proposition by induction on $b$. If $b=0$ or $b=1$, then we are done. Let $b>1$. The Taylor resolution of $S/I$ shows that $G$ has $a+b$ edges on the vertices $v_1,\dots,v_{t_{a+b}}$ such that $\beta_{a+b,t_{a+b}}(S/I(H))\neq 0$ where $H$ is the induced subgraph on these vertices. Let $W \subseteq [n]$ such that $H=G[W]$ and denote $N=\Delta[W]$ where $\Delta$ is the simplicial complex such that $I_{\Delta}=I$. Since $t_a((S/I(H))<2a$, it follows by (\ref{lem 1}) that the graph $H$ has a vetex $v$ with $\deg_H(v)\geq2$.
Let $x_1,x_2$ be two neighbors of $v$ in $H$. Clearly, $N=(N-v)\cup(N-\{x_1,x_2\})$. Set $\Delta_1=N-v$ and $\Delta_2=N-\{x_1,x_2\}$.
Consider the long exact sequence of reduced homologies
\begin{align*}
\cdots \rightarrow \widetilde{H}_j(\Delta_1\cap\Delta_2;K)\rightarrow \widetilde{H}_j(\Delta_1;K)\oplus \widetilde{H}_j(\Delta_2;K)&\rightarrow \widetilde{H}_j(N;K)\\\rightarrow \widetilde{H}_{j-1}(\Delta_1\cap\Delta_2;K)\rightarrow\cdots
\end{align*}
where $j=t_{a+b}-(a+b)-1$.\\
If $\widetilde{H}_{j}(\Delta_1)\neq 0$, then $\beta_{a+b-1,t_{a+b}-1}(S/I(H))\neq0$ and so,
$$t_{a+b}(S/I)=t_{a+b}(S/I(H))\leq t_{a+b-1}(S/I(H))+1.$$ Our induction hypothesis implies that
\[
t_{a+b}(S/I)\leq t_a(S/I(H))+\lceil \frac{3(b-1)}{2} \rceil+1\leq t_a(S/I)+\lceil \frac{3b}{2}\rceil.
\]
So we may assume that $\widetilde{H}_{j}(\Delta_1)=0$.\\
If $\widetilde{H}_{j}(\Delta_2)\neq 0$, then $\beta_{a+b-2,t_{a+b}-2}(S/I(H))\neq0$, and so

$$t_{a+b}(S/I)=t_{a+b}(S/I(H))\leq t_{a+b-2}(S/I(H))+2.$$
If $b=2$, then
$$t_{a+2}(S/I)\leq t_a(S/I(H))+2\leq t_a(S/I))+2.$$ Let $b>2$.
The induction hypothesis implies that
\[
t_{a+b}(S/I)\leq t_a(S/I(H))+\lceil \frac{3(b-2)}{2} \rceil+2\leq t_a(S/I)+\lceil \frac{3b}{2}\rceil.
\]
So we may assume also that $\widetilde{H}_{j}(\Delta_2)=0$. Since $\widetilde{H}_j(N;K)\neq 0$, it follows that $\widetilde{H}_{j-1}(\Delta_1\cap\Delta_2;K)\neq 0$, and so $\beta_{a+b-2,t_{a+b}-3}(S/I(H))\neq0$. Then
$$t_{a+b}(S/I)=t_{a+b}(S/I(H))\leq t_{a+b-2}(S/I(H))+3.$$
If $b=2$, then
$$t_{a+2}(S/I)\leq t_a(S/I(H))+3\leq t_a(S/I))+3.$$ Let $b>2$.
Our induction hypothesis implies that
\[
t_{a+b}(S/I)\leq t_a(S/I(H))+\lceil \frac{3(b-2)}{2} \rceil+3\leq t_a(S/I)+\lceil \frac{3b}{2}\rceil.
\]
This completes the proof.
\end{proof}

\begin{theorem}\label{thm1}
Over any field, if $I=I(G)$ is the edge ideal of a graph $G$ and $t_b\geq \lceil \frac{3b}{2} \rceil$ for some $b\geq 0$, then for all $a\geq0$, $$t_{a+b}\leq t_a+t_b.$$
\end{theorem}

\begin{proof}
If $a=0$ or $a=1$, then we are done. Assume that $a\geq 2$. If $t_a=2a$, then
\[
t_{a+b}\leq at_1+t_b\leq 2a+t_b=t_a+t_b.
\]
So assume that $t_a<2a$. By (\ref{pro1}),
\[
t_{a+b}\leq t_a+\lceil \frac{3b}{2} \rceil\leq t_a+t_b,
\]
as desired.
\end{proof}

\begin{theorem}\label{a=4}
Over any field, if $I=I(G)$ is the edge ideal of a graph $G$, then for all $a\geq0$,
$$t_{a+4}\leq t_a+t_4.$$
\end{theorem}

\begin{proof}
By (\ref{thm1}) we may assume $t_4=5$ (note that $5\leq t_4\leq 8$). If $t_2=4$, then combined with $t_4=5$, lemma (\ref{thm-vanishing}) says $\beta_{3+k,5+k}(S/I)=0$ for any $k>0$. Further, by lemma (\ref{thm-corner}) we conclude that $t_k\leq k+1$ for all $k\geq 4$. So $$t_{a+4}\leq a+5\leq t_a+t_4$$ as desired.  Thus, we may assume $t_2=3$ and then $G$ has no two disjoint edges which form an induced subgraph. Similarly, we may assume $t_3=4$. If not then by (\ref{thm-corner}), $t_3=5$ and then $\beta_{3+k,5+k}(S/I)=0$ for any $k>0$ by (\ref{thm-vanishing}). Also we obtain that $t_{a+4}\leq a+5\leq t_a+t_4$.

By (\ref{thm-abederan}, we may also assume $a\geq4$. If $t_{a+4}\leq a+6$, then we are done:
\[
t_a+t_4\geq a+1+5=a+6\geq t_{a+4}.
\]
Let $t_{a+4}\geq a+7$.
The Taylor resolution of $S/I$ shows that $G$ has $a+4$ edges on the vertices $v_1,\dots,v_{t_{a+4}}$ such that $\beta_{a+4,t_{a+4}}(S/I(H))\neq 0$ where $H$ is the induced subgraph on these vertices. Let $W \subseteq [n]$ such that $H=G[W]$ and denote $N=\Delta[W]$ where $\Delta$ is the simplicial complex such that $I_{\Delta}=I$. Let $W'\subset W$ such that $|W'|=5$ and $H'=G[W']$ is the (induced) subgraph of $H$ with

\begin{equation}\label{eq1}
 \small t_2(S/I(H'))=3,~t_3(S/I(H'))=4~,t_4(S/I(H'))=5,~\mathrm{and}~\deg_{H'}(v)\geq1.
\end{equation}

First, we claim that there exists a vertex $v$ in $H'$ such that $\deg_{H'}(v)\geq3$. If not, then $\deg_{H'}(v)\leq 2$ and so $H'$ is either an induced $C_5$, or a path with $4$ edges. This contradicts (\ref{eq1}). Second, we show that the graph $H$ has a vertex of degree at least $4$. Assume on contrary that $\deg_{H}(w)\leq3$ for all $w$ in $H$. After suitable renumbering of the vertices, $H'$ has the following proper subgraph:

\begin{figure}[h]
\begin{tikzpicture}[line cap=round,line join=round,>=triangle 45,x=1cm,y=1cm]
\clip(-1.8565123791498606,0.7628499658835928) rectangle (7.659494782255121,4.9409718601879975);
\draw [line width=1.2pt] (3,4)-- (1.3933025525903067,2.274950487304406);
\draw [line width=1.2pt] (3,4)-- (4.547515163294998,2.2220496468733);
\draw [line width=1.2pt] (4.547515163294998,2.2220496468733)-- (4.309461381355021,1.1375824180356229);
\draw (2.983269388095954,4.442868360333201) node[anchor=north west] {1};
\draw (1.0726335752201102,2.19025104009435) node[anchor=north west] {2};
\draw (2.968400626906259,2.175382278904655) node[anchor=north west] {3};
\draw [line width=1.2pt] (3,4)-- (2.9869403705773725,2.3543017479510655);
\draw (4.7898238726439315,2.3017667490170655) node[anchor=north west] {4};
\draw (4.4701455070654825,1.1345689956259775) node[anchor=north west] {5};
\begin{scriptsize}
\draw [fill=black] (3,4) circle (2.5pt);
\draw [fill=black] (1.3933025525903067,2.274950487304406) circle (2.5pt);
\draw [fill=black] (4.547515163294998,2.2220496468733) circle (2.5pt);
\draw [fill=black] (4.309461381355021,1.1375824180356229) circle (2.5pt);
\draw [fill=black] (2.9869403705773725,2.3543017479510655) circle (2.5pt);
\end{scriptsize}
\end{tikzpicture}
\caption{}
\label{subgraph}
\end{figure}
Note that the graph in figure (\ref{subgraph}) do not satisfies (\ref{eq1}). If $\{2,4\}\in E(H)$ or $\{3,4\}\in E(H)$, then there is necessarily an edge $\{i,j\}\in E(H)$ such that $i,j\notin \{1,\dots,5\}$, since $1\leq \deg_{H}(w)\leq3$ for all $w$ in $H$ and $H$ has at least $11$ vertices. Then the two disjoint edges $\{i,j\}$ and $\{1,4\}$ form an induced subgraph. This contradicts the assumption $t_2=3$. So assume that $\{2,4\}\notin E(H)$ and $\{3,4\}\notin E(H)$. If $\{2,3\}\in E(H)$, then since $t_2=3$, $\{2,5\}\in E(H)$ or $\{3,5\}\in E(H)$. In the two cases, we must have an edge $\{i,j\}\in E(H)$ such that $i,j\notin \{1,\dots,5\}$. Such edge with $\{1,2\}$ (or $\{1,3\}$) are two disjoint edges that form an induced subgraph, a contradiction. So we may also assume that $\{2,3\}\notin E(H)$.

We conclude that the only option of the graph $H'$ is:

\begin{figure}[h]
\begin{tikzpicture}[line cap=round,line join=round,>=triangle 45,x=1cm,y=1cm]
\clip(-1.149877575667689,0.7831255493240843) rectangle (7.5010380256095655,4.581418180509909);
\draw [line width=1.2pt] (3,4)-- (1.3933025525903067,2.274950487304406);
\draw [line width=1.2pt] (3,4)-- (4.547515163294998,2.2220496468733);
\draw [line width=1.2pt] (4.547515163294998,2.2220496468733)-- (4.309461381355021,1.1375824180356229);
\draw (2.9863414461929985,4.432730568612954) node[anchor=north west] {1};
\draw (1.0736780749731054,2.175382278904653) node[anchor=north west] {2};
\draw (2.966065862752505,2.1686237510911552) node[anchor=north west] {3};
\draw [line width=1.2pt] (3,4)-- (2.9869403705773725,2.3543017479510655);
\draw (4.784109844583428,2.2970357795476155) node[anchor=north west] {4};
\draw (4.4732175651625266,1.127810467812477) node[anchor=north west] {5};
\draw [line width=1.2pt] (4.309461381355021,1.1375824180356229)-- (2.9869403705773725,2.3543017479510655);
\draw [line width=1.2pt] (4.309461381355021,1.1375824180356229)-- (1.3933025525903067,2.274950487304406);
\begin{scriptsize}
\draw [fill=black] (3,4) circle (2.5pt);
\draw [fill=black] (1.3933025525903067,2.274950487304406) circle (2.5pt);
\draw [fill=black] (4.547515163294998,2.2220496468733) circle (2.5pt);
\draw [fill=black] (4.309461381355021,1.1375824180356229) circle (2.5pt);
\draw [fill=black] (2.9869403705773725,2.3543017479510655) circle (2.5pt);
\end{scriptsize}
\end{tikzpicture}
\caption{$H'$}
\label{subgraph2}
\end{figure}
In this case, there is two distinct edges $\{i_1,j_1\}\in E(H)$ and $\{i_2,j_2\}\in E(H)$ such that $i_1,j_1,i_2,j_2\notin \{1,\dots,5\}$. Since $t_2=3$, each vertex of the set $\{2,3,4\}$ is connected to $i_1$ or $j_1$ by edge. Similarly, each vertex of the set $\{2,3,4\}$ is connected to $i_2$ or $j_2$ by edge. This contradicts the assumption $\deg_{H}(w)\leq3$ for all $w$ in $H$. We conclude that there is a vertex $v$ in $H$ such that $\deg_{H}(v)\geq4$.

Let $v\in W$ with $\deg_H(v)\geq4$. Let $x_1,x_2,x_3,x_4$ be four neighbors of $v$ in $H$. Clearly, for $N=\Delta[W]$,  $N=(N-v)\cup(N-\{x_1,x_2,x_3,x_4\})$. Set $\Delta_1=N-v$ and $\Delta_2=N-\{x_1,x_2,x_3,x_4\}$. We may assume that $\widetilde{H}_{j}(\Delta_1)\neq0$, where $j=t_{a+4}-(a+4)-1$. For otherwise, we obtain that $\beta_{a+3,t_{a+4}-1}(S/I)\neq0$ and so $$t_{a+4}\leq t_{a+3}+1\leq t_a+t_3+1=t_a+5=t_a+t_4$$ as desired.
Using the Mayer-Vietoris sequence we have $\widetilde{H}_j(\Delta_2)\neq0$ or $\widetilde{H}_{j-1}(\Delta_1\cap\Delta_2)\neq0$. If $\widetilde{H}_j(\Delta_2)\neq0$, then $\beta_{a,t_{a+4}-4}(S/I)\neq0$, and so $$t_{a+4}\leq t_{a}+4 < t_a+t_4.$$
Finally, if $\widetilde{H}_{j-1}(\Delta_1\cap\Delta_2)\neq0$, then $\beta_{a,t_{a+4}-5}(S/I)\neq0$, and so $$t_{a+4}\leq t_a+5=t_a+t_4.$$

\end{proof}

Combining theorem (\ref{a=4}) with lemma (\ref{thm-abederan}) we obtain the following.

\begin{corollary}
For any edge ideal over any field, the subadditivity condition holds for $0\leq b\leq 4$ and any  $a\geq 0$.
\end{corollary}

Finally, we state the following corollary.

\begin{corollary}
Over any field, if $I=I(G)$ is the edge ideal of a graph $G$ and $t_j\geq \lceil \frac{3j}{2} \rceil$ for all $5\leq j\leq \lfloor \frac{\mathrm{pd}_S(S/I)}{2}\rfloor$, then $I$ satisfies the subadditivity condition. In particular, the subadditivity condition holds if $\mathrm{pd}_S(S/I)\leq 9$
\end{corollary}

\begin{proof}
By lemma (\ref{thm-abederan}) and theorem (\ref{a=4}) we may assume that $a>4$, $b>4$. we also assume that $a\geq b$. Since $a+b\leq \mathrm{pd}_S(S/I)$, we obtain that $5\leq b\leq \lfloor \frac{\mathrm{pd}_S(S/I)}{2}\rfloor$. Hence the assertion follows from (\ref{thm1}).
\end{proof}

\emph{Data sharing not applicable to this article as no datasets were generated or analysed during the current study.}

\end{document}